\documentclass[a4paper,11pt]{article}
\usepackage[T1]{fontenc} 
\usepackage{amsmath,amsthm}
\usepackage[francais,english]{babel}
\usepackage[dvips]{graphicx}
\usepackage{todonotes}
\usepackage{color}
\usepackage{tikz}
\usepackage{tikz-cd}
\usetikzlibrary{arrows,calc,matrix,fit}
\usepgflibrary{arrows}
\usepackage{soul}
\usepackage{amsfonts,amssymb}
\usepackage{geometry}
\usepackage{hyperref}
\usepackage{stmaryrd}
\usepackage{fancyhdr}
\usepackage{url}
\usepackage{todonotes}
\usepackage{dsfont}
 \usepackage[utf8]{inputenc} 

\theoremstyle{definition}

\newtheorem*{TV2}{Theorem}

\newtheorem*{thm*}{Theorem}

\newtheorem{prop}{Proposition}[section]

\newtheorem{lemma}[prop]{Lemma}
\newtheorem*{lethallemma}{{Theorem \ref{lethal lemma}}}
\newtheorem{thm}[prop]{Theorem}
\newtheorem{definition}[prop]{Definition}
\newtheorem{corollary}[prop]{Corollary}

\newtheoremstyle{pourlesremarques}{\topsep}{\topsep}{\normalfont}{}{\bfseries}{.}{ }{}
\theoremstyle{pourlesremarques}
\newtheorem{rem}[prop]{Remark}
\newtheorem*{rem*}{Remark}
\newtheoremstyle{pourlesexemples}{\topsep}{\topsep}{\normalfont}{}{\bfseries}{.}{ }{}
\theoremstyle{pourlesexemples}

\def\presuper#1#2%
  {\mathop{}%
   \mathopen{\vphantom{#2}}^{#1}%
   \kern-\scriptspace%
   #2}

\newcommand{\St}{{\operatorname{St}}}
\newcommand{\cusp}{{\operatorname{cusp}}}
\newcommand{\cent}{{\operatorname{cent}}}
\newcommand{\whitt}{{\operatorname{Whitt}}}
\newcommand{\stand}{{\operatorname{stand}}}

\newcommand{\disc}{{\operatorname{disc}}}

\newcommand{\gen}{{\operatorname{gen}}}

\newcommand{\J}{\mathrm{J}}
\newcommand{\Irr}{\Pi}
\newcommand{\rep}{\mathfrak{R}}
\newcommand{\Sm}{\mathfrak{S}}

\newcommand{\Ext}{\mathrm{Ext}}

\newcommand{\Ind}{\operatorname{Ind}}

\newcommand{\Hom}{\operatorname{Hom}}
\renewcommand{\subset}{\subseteq}

\newcommand{\diag}{\operatorname{diag}}

\newcommand{\GL}{\operatorname{GL}}
\newcommand{\Sp}{\operatorname{Sp}}

\newcommand{\abs}[1]{\left|{#1}\right|}

\newcommand{\Res}{\operatorname{Res}}

\newcommand{\p}{{\mathfrak{p}}}   

\renewcommand{\o}{\mathfrak{o}}

\renewcommand{\d}{\delta}

\renewcommand{\l}{\lambda}

\newcommand{\C}{\mathbb{C}}

\newcommand{\N}{\mathbb{N}}

\renewcommand{\S}{\mathcal{S}}

\newcommand{\1}{\mathbf{1}}

\def\ess{\operatorname{ess}}
\def\GL{\operatorname{GL}}

\def\\Hom{\operatorname{\Hom}}
\def\Irr{\operatorname{Irr}}

\def\St{\operatorname{St}}

\def\diag{\operatorname{diag}}

\def\Res{\operatorname{Res}}

\newcommand{\D}{\Delta}
\renewcommand\Re{\mathrm{Re}}

\def\presuper#1#2%
  {\mathop{}%
   \mathopen{\vphantom{#2}}^{#1}%
   \kern-\scriptspace%
   #2}
\title{Extension of Whittaker functions and test vectors}
\author{R. Kurinczuk\footnote{Robert Kurinczuk, Department of Mathematics, Imperial College London, SW7 2AZ. U.K. \newline Email:~robkurinczuk@gmail.com, Tel: +44(0)7921 221967}, N. Matringe\footnote{Nadir Matringe, Universit\'e de Poitiers, Laboratoire de Math\'ematiques et Applications,
T\'el\'eport 2 - BP 30179, Boulevard Marie et Pierre Curie, 86962, Futuroscope Chasseneuil Cedex. France. \newline Email:~Nadir.Matringe@math.univ-poitiers.fr}}
\date{}
\begin{document}
\maketitle

\begin{abstract}
 We show that certain products of Whittaker functions and Schwartz functions on a general linear group extend to Whittaker functions on a larger general linear group.  This generalizes results of Cogdell--Piatetski-Shapiro \cite{CPS} and Jacquet--Piatetski-Shapiro--Shalika \cite{JPSS83}. As a consequence, we prove that the Rankin--Selberg $L$-factor of the product of a discrete series representation and the Zelevinsky dual of a discrete series representation is given by a single Rankin--Selberg integral.
\end{abstract}


\section{Introduction}
In their seminal work, Jacquet--Piatetski-Shapiro--Shalika developed the Rankin--Selberg method for automorphic representations, treating the local theory in \cite{JPSS83}. 
The local Euler factors, or \emph{$L$-factors}, are defined as greatest common divisors of families of local \emph{Rankin--Selberg} integrals.  
As a consequence of the definition, each $L$-factor can be written as a finite sum of Rankin--Selberg integrals, but it is not clear whether one can 
find \emph{test vectors} expressing the $L$-factor as a single Rankin--Selberg integral.  When the local representations are cuspidal, this is possible by an explicit 
computation~\cite{KM17}.  

In this article, we prove a local result on the extension of Whittaker and Schwartz functions (stated precisely at the end of this introduction).  We show that 
this result simultaneously generalizes results of \cite{CPS} and \cite{JPSS83}, both of which have proved useful in the theory of integral representations of $L$-factors.   
As a consequence of our result we answer the test vector question in the affirmative for the $L$-factor of the 
product of a discrete series representation and the Zelevinsky dual of a discrete series representation.  

Let $F$ be a locally compact non-archimedean local field, $\psi$ be a nontrivial character of $F$, and put $G_n=\GL_n(F)$. Let $\pi$ and $\pi'$ be irreducible smooth 
representations of $G_n$ and $G_m$ with $n\geqslant m$, and let $\S(\pi)$ and $\S(\pi')$ be the standard modules whose irreducible quotients are $\pi$ and $\pi'$ respectively.
If $n>m$, the $L$-factor $L(s,\pi,\pi')$ attached to the pair $(\pi,\pi')$ by Jacquet--Piatetski-Shapiro--Shalika \cite{JPSS83} is a finite sum $\sum_i I(s,W_i,W'_i)$ of
Rankin-Selberg integrals for Whittaker functions $W_i\in W(\S(\pi),\psi)$ and $W'_i\in W(\S(\pi'),\psi^{-1})$ in the Whittaker models of the standard modules. 
If $n=m$, it is a finite sum $\sum_i I(s,W_i,W'_i,\phi_i)$ of Rankin--Selberg integrals with $W_i,W'_i$ as before and $\phi_i\in \mathcal{C}_c^\infty(F^n)$ a Schwartz function.  %
A tuple $(W,W')$ or $(W,W',\phi)$ whose Rankin--Selberg integral equals the $L$-factor is called a \emph{test vector}. 
The systematic study of test vectors for local Rankin--Selberg integrals was initiated in the doctoral thesis \cite{KM10}, where many interesting partial results were obtained.

Let $\delta$ and $\delta'$ be discrete series representations of $G_n$ and $G_m$ respectively, and $\presuper{t}\delta'$ denote the Zelevinsky dual of $\delta'$. 
In this article we prove the following test vector result:

\begin{TV2}[Remark \ref{remLtriv}, Proposition \ref{one cusp} and Theorem \ref{TV2}]
There exist $W\in W(\delta,\psi)$ and $W'\in W(\S(\presuper{t}\delta'),\psi^{-1})$, such that, 
\begin{enumerate}
\item If $n>m$
\[L(s,\delta,\presuper{t}\delta')=I(s,W,W');\] 
\item If $n=m$,  in addition there exists a Schwartz function $\phi$ on $F^n$, such that 
\[L(s,\delta,\presuper{t}\delta')=I(s,W,W',\phi).\] 
Moreover one can always choose $\phi=\mathbf{1}_{(\p^f)^{n-1} \times (1+\p^f)}$ for $f$ large enough.\end{enumerate}
\end{TV2}

In fact,~$L(s,\d,\presuper{t}\d')=L(s,\d,\d')$, so this also shows that the Euler factor $L(s,\d,\d')$ is given by a single integral $I(s,W,W')$ or $I(s,W,W',\phi)$; however with~$W'$ inside~$W(\S(\presuper{t}\d'),\psi^{-1})$ rather than inside $W(\d',\psi^{-1})$.  Indeed, in this case, this makes the test vector question simpler
 %
 %
 as~$W(\S(\presuper{t}\d'),\psi^{-1})$ contains~$W(\d',\psi^{-1})$ as a proper subspace. An easy, yet already instructive, example is given in Section \ref{steinberg case}, where we take $\d$ and $\d'$ to be Steinberg representations.  In the general case, we do not address the question of finding explicit test vectors, which would require other techniques, for example Bushnell--Kutzko type theory in the spirit of \cite{KM17,PS08}.  The techniques of this paper are entirely different to \cite{KM17}; here we use Bernstein and Zelevinsky's theory of derivatives, in particular Cogdell--Piatetski-Shapiro's interpretation of derivatives \cite{CPS}, to reduce to the cuspidal case.

To obtain our test vector result, the key step is Theorem \ref{lethal lemma}, which generalizes both \cite[Proposition 9.1]{JPSS83} and part of~\cite[Proposition 1.7]{CPS}.  To state Theorem \ref{lethal lemma}, first we explain a consequence of the results of \cite{CPS} which we recall and expand in Sections \ref{CPS interpretation} and \ref{extension of Whittaker}. Denote by $P_n$ the mirabolic subgroup of $G_n$ consisting of matrices with final 
row $\eta_n=\begin{pmatrix}0&\cdots&0&1\end{pmatrix}$, and by $N_n$ its standard maximal unipotent subgroup. If $\tau$ is a submodule of $\Ind_{N_n}^{P_n}(\psi)$ 
such that the $k$-th Bernstein-Zelevinsky derivative $\tau^{(k)}$ has a central character, then there is a natural embedding 
\[\overline{S}:\tau^{(k)}\rightarrow\Ind_{N_{n-k}}^{G_{n-k}}(\psi),\] and we put $W(\tau^{(k)},\psi)=\overline{S}(\tau^{(k)})$. 
We can now state Theorem \ref{lethal lemma} as:

\begin{lethallemma}
Let $n >k \geqslant 1$. For any Schwartz function $\phi$ on $F^{n-k}$ and any $W_0\in W(\tau^{(k)},\psi)$, the
map $|\mathrm{det}(g)|^{k/2}W_0(g)\phi(\eta_{n-k} g)$ extends to a function in $\tau\subset \Ind_{N_n}^{P_n}(\psi)$. 
\end{lethallemma}
In particular, this applies to the case where $\tau$ is a submodule of the space of restrictions to $P_n$ of functions in the Whittaker model of a representation of Whittaker type of $G_n$. 

\section*{Acknowledgements}
This work was started during a research visit of the second author (N.M.) to Imperial College London and he would like to thank them for their hospitality. 
The visit was supported by the LMS (Research in Pairs Grant) and by GDRI: Representation Theory (2016-2020).  The authors thank David Helm, Gil Moss, Dipendra Prasad and Shaun Stevens for fruitful conversations. Most importantly, we thank the referee for pointing out a mistake
in a computation of a previous version, his precise reading and accurate corrections, and his very 
useful suggestions concerning the presentation of the paper.
The second author was supported by the grant ANR-13-BS01-0012 FERPLAY.

\section{Notation and Preliminaries}\label{notation}
As in the introduction, $F$ denotes a locally compact non-archimedean local field, we put $G_n=\GL_n(F)$, and $\psi$ denotes a nontrivial character of $F$.  We denote by $|\,\cdot\,|$ the normalized absolute value of $F$, by $\o$ its ring of integers, and by $\p$ the maximal ideal of $\o$. 
We let $q$ denote the order of the residue field $\o/\p$, $Z_n$ denote the centre of $G_n$, and $K_n=\GL_n(\o)$. We denote by $\nu$ the character of $G_n$ given by $\nu(g)=\abs{\det(g)}$ for 
$g\in G_n$. Let $P_n$ be the \textit{mirabolic subgroup} of $G_n$, consisting of all matrices in $G_n$ with bottom row $\eta_n=(0,\dots,0,1)$. Any 
$p\in P_n$ can be written 
in a uniquely as $p=g(p)u(p)$ for $g(p)\in G_{n-1}$, and $u(p)\in U_n$ the unipotent radical of the standard (block upper triangular) parabolic 
subgroup of $G_n$ of type $(n-1,1)$.  For $p\in P_n$, we put $\d_{P_n}(p)=\nu(g(p))$. We have the \emph{Iwasawa decomposition} \[G_n=P_nZ_nK_n,\] and if $g=pzk$ 
in this decomposition, then $\nu(z)$ only depends on $g$, and we 
shall write $\nu(z(g))$ for $\nu(z)$. We consider $G_k$ as a subgroup of $G_n$ for $1\leqslant k \leqslant n$ via the embedding 
$g\mapsto \diag(g,I_{n-k})$.  



Let $G$ be a locally compact totally disconnected group.  By a \emph{representation} of $G$ we mean a smooth representation on a complex vector space. We denote by $\Sm(G)$ the category 
of (smooth) representations of $G$, and by $\rep(G)$ the category of finite length representations of $G$. When practical, we use the same notation for the collection 
of objects in a category and the underlying category; so, for example, $\pi\in\Sm(G)$ will mean $\pi$ is an object of $\Sm(G)$.  

We let $G_0$ denote the trivial group. For $G=G_n$ we set 
$\Sm(n)=\Sm(G_n)$, $\rep(n)=\rep(G_n)$, $\Sm =\bigsqcup_{n=0}^\infty \Sm(n)$ and $\rep=\bigsqcup_{n=0}^\infty \rep(n)$. We also denote by  $\Irr(n)$ the collection of 
irreducible representations of $G_n$, and set $\Irr=\bigsqcup_{n=0}^\infty \Irr(n)$. 

For $H$ a closed subgroup of a locally compact totally disconnected group $G$, we use the notation $\Ind_H^G:\Sm(H)\rightarrow\Sm(G)$ for the functor 
of normalized induction.  For representations $\pi_i$ of $G_{n_i}$, $i=1,\dots,t$ we denote by $\pi_1\times\cdots \times \pi_t$ the 
representation of $G_{n_1+\cdots +n_t}$ obtained from $\pi_1\otimes\cdots \otimes \pi_t$ by normalized 
parabolic induction for the standard parabolic of type $(n_1,\cdots,n_t)$. For a representation $\pi$ and a character 
$\chi$ of $G_n$ let $\chi\pi$ be the representation on the space of $\pi$ given by $(\chi\pi)(g)=\chi(g)\pi(g)$ for $g\in G_n$. 

 In \cite[Section 3]{BZ77}, the authors define functors 
 \begin{align*}
 \Phi^-:\Sm(P_n)&\rightarrow \Sm(P_{n-1});\\
\Psi^-:\Sm(P_n)&\rightarrow \Sm(G_{n-1}). 
\end{align*}
(in fact, such functors originally appear in \cite{BZ76}, but we use the modified definition of \cite{BZ77}). It is shown in \cite{BZ77} that $\Phi^-,\Psi^-$ are exact, and restrict as functors to
\begin{align*}\Phi^-:\rep(P_n)&\rightarrow \rep(P_{n-1});\\\Psi^-:\rep(P_n)&\rightarrow \rep(G_{n-1}).\end{align*} It is also shown for a representation $\pi\in \Sm(G_n)$, that $\pi\in\rep(G_n)$ if 
and only if $\pi\vert_{P_n}\in\rep(P_n)$. Following 
\cite{BZ77}, for $\pi\in \Sm(P_n)$, and $k\geqslant 1$, we set \[\pi^{(k)}=\Psi^-(\Phi^-)^{k-1} \pi, \qquad \pi^{(0)}=\pi.\] 
For $\pi\in \Sm(G_n)$, and $k\geqslant 0$, we set $\pi^{(k)}=(\pi\mid_{P_n})^{(k)}$. 

Let $N_n$ be the group of upper triangular unipotent matrices in $G_n$ and, by abuse of notation, let $\psi$ also denote the character on $N_n$ defined by
\[
\psi(u)=\psi(u_{1,2}+\cdots +u_{n-1,n}),
\]
for $u\in N_n$ with $(i,j)$-th entry $u_{i,j}$.

By Frobenius reciprocity and \cite[Proposition 3.2]{BZ77}, for $\pi\in \Sm(P_n)$, we have
\[\Hom_{N_n}(\pi,\psi)\simeq \Hom_{P_n}(\pi,\Ind_{N_n}^{P_n}(\psi)) \simeq \Hom_{\C}(\pi^{(n)},\C).\] 
In particular when $\pi$ belongs to $\rep(P_n)$, the (finite) dimension of $\pi^{(n)}$ is precisely that of the space $\Hom_{N_n}(\pi,\psi)$ of \emph{Whittaker functionals} on $\pi$. Following the authors of \cite{JPSS83}, we introduce the following classes of representations.

\begin{definition}\label{Whittaker type}
We say that a representation $\pi\in \rep(P_n)$ is of \emph{Whittaker type} if $\Hom_{N_n}(\pi,\psi)$ is of dimension one. 
We say that a representation $\pi\in \rep(G_n)$ is of \emph{Whittaker type} if $\pi\mid_{P_n}$ is of Whittaker type, i.e. if 
$\Hom_{N_n}(\pi,\psi)$ is of dimension one.  
\end{definition}

In either case, a finite length representation $\pi$ is of Whittaker type if and only if $\pi^{(n)}\simeq \C$. In particular, 
if $\pi_1$ and $\pi_2$ are both representations of Whittaker type of $G_{n_1}$ and $G_{n_2}$ respectively, then 
the representation $\pi_1\times \pi_2$ is also of Whittaker type according to \cite[Corollary 4.14, c)]{BZ77}. 
Let $\pi$ be a representation of Whittaker type of 
$P_n$ (resp. $G_n$). By Frobenius reciprocity, 
there is a unique up to scalar nonzero intertwining operator from $\pi$ to $\Ind_{N_n}^{P_n}(\psi)$ 
(resp. $\Ind_{N_n}^{G_n}(\psi)$), and we denote by $W(\pi,\psi)$ the image of $\pi$ and call it the \textit{Whittaker model} of $\pi$, though it is not always a \emph{model} of $\pi$, i.e.~in general $W(\pi,\psi)$ 
is a quotient of $\pi$, but not isomorphic to it.

An irreducible representation of Whittaker type of $G_n$ is called \emph{generic}.  In fact, by \cite{GK71}, the generic representations of $G_n$ are those 
irreducible representations $\pi$ such that $\Hom_{N_n}(\pi,\psi)\neq 0$. By exactness of the $n$-th derivative functor, a representation of $G_n$ of Whittaker type 
has a unique generic subquotient.  For the group $G_n$, it follows from \cite[Lemma 4.5]{BZ77} and \cite[Section 9]{Z80}, that if $\d_1,\dots,\d_t$ are
irreducible \textit{essentially square integrable}, which we call \textit{discrete series}, representations, then 
$\d_1\times \cdots \times \d_t$ is a representation of Whittaker type.

If $\d$ is an irreducible discrete series representation, we denote by $e(\d)$ the unique real number such that $\nu^{-e(\d)}\d$ is unitary. We say that a representation in $\rep$ is 
a \textit{standard module} if it is of the form $\d_1\times \cdots \times \d_t$, with $\d_i$'s irreducible discrete series such that 
$e(\d_1)\geqslant \dots \geqslant e(\d_t)$ (these representations are called \emph{induced of Langlands type} in \cite{JPSS83} and \cite{JS83}). If all the $\d_i$'s are cuspidal, we say that the standard module $\d_1\times \cdots \times \d_t$ is \textit{cuspidally induced}. By \cite{S78}, a 
standard module $\S$ has a unique irreducible quotient $\pi(\S)$, and the map 
\[\S\in (\rep_\stand(n)/\simeq) \mapsto \pi(\S)\in (\Irr(n)/\simeq)\] is a bijection; the inverse of this bijection we denote by 
$\pi\mapsto \S(\pi)$. We call $\S(\pi)$ the standard module over $\pi$. 

We will use the following lower index notation. If $A$ is one of the collections $\Sm, \rep$ or $\Irr$, and 
$\bullet$ is an abbreviation of a type of representation, we will denote by $A_\bullet$ the collection of representations of type $\bullet$ inside $A$. If a representation $\pi$ in $A$ has type $\bullet_1$ and $\bullet_2$ together, we will write 
$\pi\in A_{\bullet_1,\bullet_2}$. We will denote by $A_{\bullet}(n)$ the representations of $G_n$ in $A_\bullet$.
We use the following abbreviations of types of representations:

\begin{itemize}
 \item $=~\cusp$: cuspidal, 
 \item $=~\disc$: essentially square integrable/discrete series, 
 \item $=~\gen$: generic, 
 \item $=~\stand$: standard module, 
 \item $=~\stand-\cusp$: cuspidally induced standard module, 
 \item $=~\whitt$: Whittaker type, 
 \item $=~\cent$: with central character.
\end{itemize}

For example, we have the well known inclusions 
\begin{align*}
&\Irr\subset \Sm_\cent,  &&\Irr_\cusp \subset \Irr_\disc \subset \Irr_\gen\subset\rep_{\whitt,\cent}, \\
&\rep_\stand \subset \rep_{\whitt,\cent}&& \Irr\cap \rep_\stand =\Irr \cap \rep_\whitt= \Irr_{\gen}.
\end{align*}
 Notice that $\rep_{\stand-\cusp}$ is different from  $\rep_{\stand,\cusp}=\Irr_\cusp$.
Recall that
\[
(\Irr_\disc/\simeq)=\{\St_k(\rho):\rho\in(\Irr_\cusp/\simeq),\,k\in\N\}
\]
where $\St_k(\rho)$ is the up to nonzero scalar unique isomorphism unique irreducible quotient of $\nu^{(1-k)/2}\rho\times \nu^{(3-k)/2}\rho\times\cdots\times\nu^{(k-1)/2}\rho$ (see \cite[Theorem 9.3]{Z80}). By \cite[Secion 9.1]{Z80}, the representation $\St_k(\rho)$ is also the unique irreducible submodule of 
\[\S_k(\rho)=\nu^{(k-1)/2}\rho\times \nu^{(k-3)/2}\rho\times\cdots\times \nu^{(1-k)/2}\rho\in \rep_{\stand},\] and 
we denote by $\Sp_k(\rho)$ its irreducible quotient, i.e. $\S_k(\rho)=\S(\Sp_k(\rho))$. The author of \cite{Z80} attaches to 
$\d=\St_k(\rho)$ the \emph{cuspidal segment} $\Delta=[\nu^{\frac{1-k}{2}}\rho,\dots,\nu^{\frac{k-1}{2}}\rho]$ (see \cite[Sections 3 and 9]{Z80}), and defines \emph{being linked} and \emph{preceding} relations on the set of cuspidal segments in \cite[Section 4.1]{Z80}. According to \cite[Section 9]{Z80}, if $\d=\St_k(\rho)$, then its Zelevinsky dual is $^t\d=\Sp_k(\rho)$  (see \cite[Section 9]{Z80} for the definition of the Zelevinsky dual).   An easy induction on the number of segments using the fact that discrete series corresponding to unlinked segments commute in the sense of parabolic induction shows the well known fact that $\pi$ is in $\rep_\stand$ if and only if it is isomorphic to a representation of the form $\d_1\times \dots \times \d_t$ with $\d_i$ discrete series representations such that $\D_i$ does not precede $\D_j$ for $i<j$. 
By \cite[Theorem 9.7]{Z80}, $\pi$ belongs to $\Irr_{\gen}$ if and only if it can be written as a (necessarily commutative) product of discrete series $\pi=\d_1\times \dots \times \d_t$ where the corresponding $\D_i$'s are unlinked, the $\d_i$'s being unique up to reordering. In particular $\Irr\cap \rep_\stand = \Irr_{\gen}$ as we mentioned before and $\S(\pi)=\pi$ when $\pi$ is generic.

We denote by $\mathcal{C}_c^\infty(F^n)$ the space of smooth functions on $F^n$ with compact support. 

\section{Whittaker models and derivatives}

Here we recall some useful but not very well-known facts from \cite{EH14} and \cite{CPS} about Whittaker models and the interpretation of derivatives in the space of Whittaker functions. Then we push the techniques of \cite{CPS} in the spirit of \cite{M13} to obtain our first main result, which is Theorem \ref{lethal lemma} about extending Whittaker functions to larger linear groups.

\subsection{The Whittaker model of representations of Whittaker type}\label{whittaker type}

Here we highlight a result of \cite{EH14} on Whittaker models of representations of Whittaker type. We shall only need it in order to prove that
our first main result Theorem \ref{lethal lemma} extends \cite[Proposition 9.1]{JPSS83}. However, the result was not known to us until recently, and 
it is a quite striking result on representations of Whittaker type. It says that the Whittaker model of such a representation is a submodule of the 
Whittaker model of a standard module. Notice that by the main result of \cite{JS83}, the Whittaker model of a standard module is 
isomorphic to this standard module and it moreover  has a Kirillov model. The result follows from 
\cite[Lemma 3.2.4]{EH14} and \cite[Lemma 4.3.9]{EH14}. The setting of \cite{EH14} being much more general and the result being stated there with a different terminology, 
we give a largely self-contained proof here, which in any case uses the ideas of \cite{EH14}. 

If $\pi$ belongs to $\rep_\whitt$, we denote by $\pi^\gen$ its unique generic subquotient. Notice that 
\[\pi^\gen\simeq W(\pi^\gen,\psi) \subset W(\pi,\psi),\] and that $W(\pi^\gen,\psi)$ is the unique irreducible submodule of 
$W(\pi,\psi)$. We give a different classification of generic representations, due to Zelevinsky as well. 

\begin{prop}\label{Zgen}
Let $\pi\in \Irr_\gen(n)$, then it is the unique irreducible submodule of a cuspidally induced standard module $\S_c(\pi)\in \rep_{\stand-\cusp}(n)$. 
All cuspidally induced standard modules containing $\pi$ are isomorphic.
\end{prop}
\begin{proof}
By \cite[Theorem 6.1]{Z80}, with the same notations, $\pi$ is of the form $\langle a\rangle$ for $a$ a sequence of cuspidal segments satisfying 
the non-preceding ordering condition (see \cite{Z80} for the precise statement). However, the representation $\pi$ being generic, 
its highest derivative is $\pi^{(n)}\simeq \C$. By~\cite[Theorem 6.1]{Z80}, this implies that the cuspidal segments occuring in $a$ are all 
cuspidal representations, and this exactly says that $\pi$ is the unique irreducible submodule of a cuspidally induced standard module. It is also a
consequence of \cite[Theorem 6.1]{Z80} that all cuspidally induced standard modules containing $\pi$ are isomorphic
\end{proof}

If $\pi$ is of Whittaker type, we set $\S_c(\pi)=\S_c(\pi^\gen)$, we will soon show that $W(\pi,\psi)\subset W(\S_c(\pi),\psi)$, as a consequence of the following proposition.

\begin{lemma}\label{ext}
Let $\pi\in \rep_\whitt(n)$, if $\tau\in \Irr(n)$ admits a nontrivial extension by $\S_c(\pi)$, then $\tau=\pi^\gen$. 
\end{lemma}
\begin{proof}
Set $G=G_n$. Write $\S_c(\pi)=\rho_1\times \dots \times \rho_t$ with the cuspidal segments $[\rho_i]$ satisfy the 
non-preceeding condition described at the end of Section \ref{notation}, we set $M$ to be the standard Levi subgroup of $G$, such that 
\[\rho=\rho_1\otimes \dots \otimes \rho_t\] is a representation of $M$. We denote by $\J^G_M$ the normalized Jacquet functor from 
$\Sm(G)$ to $\Sm(M)$. Suppose that $\tau$ admits a nontrivial extension by $\S_c(\pi)$, this means that 
\[\Ext_G^1(\tau,\S_c(\pi))\neq 0,\]
(see \cite{C81} for the basic definitions concerning $\Ext$-functors in the categories $\Sm(G)$ and  $\Sm(M)$).  Then by \cite[Theorem A.12]{C81}, this is equivalent to \[\Ext_M^1(\J_G^M(\tau),\rho)\neq 0.\] We claim that this implies there is an irreducible subquotient $\mu$ of 
$\J_G^M(\tau)$ (which has finite length) which admits a nontrivial extension by $\rho$. Let's justify this claim by writing an exact sequence 
\[0\rightarrow \J_1 \rightarrow \J_G^M(\tau) \rightarrow \J_2\rightarrow 0\] with $\J_2$ irreducible. Then writing the long exact sequence 
\begin{center}
\begin{tikzpicture}[descr/.style={fill=white,inner sep=1.5pt}]
\matrix (m) [matrix of math nodes, column sep=2.5em, row sep=1.5em,text height=1.5ex, text depth=0.25ex, ]
{
 {0}& { \Hom_M(\J_2,\rho) }& { \Hom_M(\J_G^M(\tau),\rho)} & {\Hom_M(\J_1,\rho)}\\
& { \Ext_M^1(\J_2,\rho)}&{ \Ext_M^1(\J_G^M(\tau),\rho) }& { \Ext_M^1(\J_1,\rho)}&{\cdots,}\\
};
       \path[->, scale=4,overlay]
        (m-1-1) edge (m-1-2)
        (m-1-2) edge (m-1-3)
         (m-1-3) edge (m-1-4)
        (m-1-4) edge[out=355,in=175]  (m-2-2)
        (m-2-2) edge (m-2-3)
        (m-2-3) edge (m-2-4)
        (m-2-4) edge (m-2-5);
\end{tikzpicture}
\end{center}
%
this implies that $\Ext_M^1(\J_2,\rho)$ and $\Ext_M^1(\J_1,\rho)$ can not be zero together because 
$\Ext_M^1(\J_G^M(\tau),\rho)$ is not. If $\Ext_M^1(\J_2,\rho)\neq 0$ then we set $\mu=\J_2$, otherwise we repeat the same operation with $\J_1$ instead of $\J_G^M(\pi)$ (as $\J_G^M(\pi)$ is of finite length, this process terminates to find an irreducible $\mu$).   Now by \cite{B84}, we must have $\mu\simeq \chi\rho$ for $\chi$ an unramified character of $G$ because 
$\rho$ is cuspidal, but then 
$\mu\simeq \rho$ by cuspidality of $\tau$ and $\rho$ (otherwise they would be in direct sum by \cite[Theorem 2.4. (b)]{BZ77}). Hence 
$\rho$ is a subquotient of $\J_G^M(\tau)$, so it is a quotient of it as well (\cite[Theorem 2.4. (b)]{BZ77}). 
Thus $\Hom_G(\tau,\S_c(\pi))\simeq \Hom_M(\J_G^M(\tau),\rho)\neq 0$, and hence $\tau$ is the unique irreducible submodule of $\S_c(\pi)$, 
i.e. $\tau\simeq \pi^\gen$.  
\end{proof}

We denote by $\Res_{P_n}$ the map from $\Ind_{N_n}^{G_n}(\psi)$ to $\Ind_{N_n}^{P_n}(\psi)$ which is the restriction of functions to $P_n$.
By the main result of \cite{JS83}, if $\pi\in \rep_{\stand}(n)$, then $\pi\simeq W(\pi,\psi)$ and $\Res_{P_n}$ is injective on 
$W(\pi,\psi)$. We will now show that the last part of this result remains true for $\pi\in \rep_\whitt$, i.e. that if $\pi\in \rep_\whitt$, then 
$W(\pi,\psi)$ has a Kirillov model.

\begin{prop}\label{kirillov}[\cite[Lemma 3.2.4]{EH14}, \cite[Lemma 4.3.9]{EH14}]
Let $\pi\in \rep_\whitt(n)$ with $n\geqslant 2$, then $W(\pi,\psi)\subset W(\S_c(\pi),\psi)$, in particular $\Res_{P_n}$ is injective on 
$W(\pi,\psi)$.
\end{prop}
\begin{proof}
Consider $W(\S_c(\pi),\psi)\subset W(\S_c(\pi),\psi)+W(\pi,\psi)\subset \Ind_{N_n}^{G_n}(\psi)$. If $W(\pi,\psi)$ was not contained in 
$W(\S_c(\pi),\psi)$, considering a Jordan-H\"older sequence of $\frac{W(\S_c(\pi),\psi)+W(\pi,\psi)}{W(\S_c(\pi),\psi)}$
, there would exist a $G_n$-module $V$ such that $W(\S_c(\pi),\psi)\subset V \subset W(\S_c(\pi),\psi)+W(\pi,\psi)$, and 
such that $\tau=\frac{V}{W(\S_c(\pi),\psi)}$ is irreducible. If the extension $V$ of $\tau$ by $\S_c(\pi)$ was trivial, 
then $\tau\subset V \subset \Ind_{N_n}^{G_n}(\psi)$ would be generic. If nontrivial, by Lemma \ref{ext}, 
this would also imply that $\tau$ is generic. This is absurd as $\tau$ would be a generic subquotient of $W(\pi,\psi)$ different from 
$W(\pi^{\gen},\psi)$, because $W(\pi^{\gen},\psi)$ is contained in $W(\S_c(\pi),\psi)$, contradicting multiplicity $1$ for $W(\pi,\psi)$. As $\S_c(\pi)$ is a standard module, the restriction map $\Res_{P_n}$ is injective of on 
$W(\S_c(\pi),\psi)$, hence on $W(\pi,\psi)$.
\end{proof}

\subsection{The Cogdell--Piatetski-Shapiro interpretation of derivatives}\label{CPS interpretation}

We first recall two important results from \cite[Section 1]{CPS}. Then, thanks to those results and those of Section \ref{whittaker type}, we can interpret \cite[Proposition 9.1]{JPSS83} as the inclusion of the Whittaker model of a representation of Whittaker type $\pi_2$ in the appropriate derivative of the Whittaker model of $\pi_1\times \pi_2$, 
for $\pi_1$ another representation of Whittaker type. Notice that according to \cite[Proposition 3.2, (f)]{BZ77}, the map 
$\Phi^-$ sends the representation $\Ind_{N_n}^{P_n}(\psi)$ surjectively onto $\Ind_{N_{n-1}}^{P_{n-1}}(\psi)$. The first result from \cite{CPS} that we shall need 
is the following observation, which is an immediate consequence of the proof of \cite[Proposition 1.1]{CPS}.

\begin{prop}\label{Phi}
For $n\geqslant 3$, the map $\Phi^-:\Ind_{N_n}^{P_n}(\psi)\rightarrow \Ind_{N_{n-1}}^{P_{n-1}}(\psi)$ identifies with 
the twisted restriction map $\nu^{-\frac{1}{2}}\Res_{P_{n-1}}$, where $\Res_{P_{n-1}}$ is the retriction of functions to $P_{n-1}$. 
\end{prop}

The second result is deeper. It is a consequence of the proof of \cite[Proposition 1.6]{CPS} (see \cite[Corollary 2.1]{M13} for the precise statement and its proof). 

\begin{prop} \label{Psi1}
Take $n\geqslant 2$, and let $\tau$ be a submodule of $\Ind_{N_n}^{P_n}(\psi)$ such that $\tau^{(1)}$ has a central character $c$. 
If $W\in \tau$, then for all $g\in G_{n-1}$, the following limit \[S(W)(g)=\lim_{z\in F^\times,\ z\to 0} \ c^{-1}(z)|z|^{\frac{(1-n)}{2}}\nu(g)^{-\frac{1}{2}}W\begin{pmatrix} zg & \\ & 1\end{pmatrix}\] 
exists (in fact the function of $z$ the limit of which is considered above is constant when $z$ tends to zero). The linear map 
$S:\tau\rightarrow \Ind_{N_{n-1}}^{G_{n-1}}(\psi)$ descends to $\tau^{(1)}$ (i.e. the kernel of $S$ contains that of $\Psi^-$, hence induces 
a linear map on $\tau^{(1)}$), and its descent $\overline{S}$ induces an isomorphism of $G_{n-1}$-modules 
between $\tau^{(1)}$ and $\overline{S}(\tau^{(1)})\subset \Ind_{N_{n-1}}^{G_{n-1}}(\psi)$.
\end{prop}

In fact, we shall use the above result later. For the moment we rather need the following, which is part of the proof of 
\cite[Corollary 2.1]{M13}, and is a kind of converse to Proposition \ref{Psi1}.

\begin{prop} \label{Psi2}
Let $c$ be a character of $Z_{n-1}$ with $n\geqslant 2$. Define $\Ind_{N_n}^{P_n}(\psi)_c\subset \Ind_{N_n}^{P_n}(\psi)$ to be the $P_n$-submodule of $\Ind_{N_n}^{P_n}(\psi)$ the elements of which are the functions $W$ such that for all $g\in G_{n-1}$, the quantity 
\[f_{c,W,g}(z)=c^{-1}(z)|z|^{\frac{(1-n)}{2}}\nu(g)^{-\frac{1}{2}}W \begin{pmatrix} zg & \\ & 1\end{pmatrix}\] becomes constant for $z$ in a punctured neighbourhood of zero in $F^\times$. 
The map \[S:\tau\rightarrow \Ind_{N_{n-1}}^{G_{n-1}}(\psi)\] defined by 
$S(W)(g)=\underset{z\in F^\times,\ z \to 0}{\lim}\  f_{c,W,g}(z)$ descends to $(\Ind_{N_n}^{P_n}(\psi)_c)^{(1)}$, and its descent $\overline{S}$ induces an isomorphism of $G_{n-1}$-modules 
between $(\Ind_{N_n}^{P_n}(\psi)_c)^{(1)}$ and 
\[S(\Ind_{N_n}^{P_n}(\psi)_c)=\overline{S}((\Ind_{N_n}^{P_n}(\psi)_c)^{(1)})\subset \Ind_{N_{n-1}}^{G_{n-1}}(\psi).\]
\end{prop}

Notice that by \cite[Proposition 4.13, (a) and (c)]{BZ77} 
(which is more precise than the more commonly used \cite[Corollary 4.14]{BZ77} describing the subquotients of the Bernstein-Zelevinsky filtration), if $\pi_1$ and $\pi_2$ are representations of $G_{n_1}$ and $G_{n_2}$ such that $\pi_1^{(n_1)}\simeq \C$, then 
$\pi_2\subset (\pi_1\times \pi_2)^{(n_1)}$. Here we give another result of this type, which by \cite{JS83} is equivalent to it if $\pi_1$ and $\pi_1\times \pi_2$ are standard modules. It is essentially a reformulation of 
\cite[Proposition 9.1]{JPSS83} using the above interpretation of derivatives.

\begin{prop}\label{inclusion of whittaker model}
Let $\pi_1$ and $\pi_2$ be representations in $\rep_\whitt(n_1)$ and $\rep_\whitt(n_2)$ respectively (with $n_1,n_2\geqslant 1$), then 
\[W(\pi_2,\psi)\subset W(\pi_1\times \pi_2,\psi)^{(n_1)}.\]
\end{prop}
\begin{proof}
Set $n=n_1+n_2$, the representation $\pi_1\times \pi_2$ is of Whittaker type, and we denote by $W(\pi_1\times \pi_2,\psi)\mid_{P_n}$ the representation
$W(\pi_1\times \pi_2,\psi)$ considered as a $P_n$-module. By the second part of Proposition \ref{kirillov}, one has \[W(\pi_1\times \pi_2,\psi)\mid_{P_n}\simeq \Res_{P_n}(W(\pi_1\times \pi_2,\psi))\subset \Ind_{N_n}^{P_n}(\psi).\]
Take $W_2\in W(\pi_2,\psi)$. By \cite[Proposition 9.1]{JPSS83}, for any $\phi\in \mathcal{C}^\infty_c(F^{n_2})$, there is $W\in W(\pi,\psi)$ such that, for any $p\in P_{n_2+1}$, we have
\[W\begin{pmatrix}p & \\ & I_{n_1-1} \end{pmatrix}=W_2(g(p))\nu(g(p))^{\frac{n_1}{2}}\psi(u(p))\phi(\eta_{n_2}g(p)).\]
By Proposition \ref{Phi}, this means that the map 
\[W':p\mapsto W_2(g(p))\nu(g(p))^{\frac{1}{2}}\psi(u(p))\phi(\eta_{n_2}g(p))\] belongs to the space
\[(\Phi^-)^{n_1-1}W(\pi_1\times \pi_2,\psi)\subset \Ind_{N_{n_2+1}}^{P_{n_2+1}}(\psi).\] 
Notice that by the first part of Proposition \ref{kirillov}, the representation $W(\pi_2,\psi)$ has a central character $c$. 
If we choose $\phi$ to be equal to $1$ in a neighbourhood of $0$ (whatever this neighbourhood is), 
then $W'$ belongs to $\Ind_{N_{n_2+1}}^{P_{n_2+1}}(\psi)_c$: indeed for $g\in G_{n_2}$, and $z\in F^\times$, one has 
\[W'\begin{pmatrix}zg & \\ & 1\end{pmatrix}=W_2(zg)\nu(zg)^{\frac{1}{2}}\phi(\eta_{n_2}zg)= c(z)|z|^{\frac{n_2}{2}} W_2(g)\nu(g)^{\frac{1}{2}}\phi(\eta_{n_2}zg),\] and the assertion follows by definition of 
$\Ind_{N_{n_2+1}}^{P_{n_2+1}}(\psi)_c$ thanks to our hypothesis on $\phi$. Moreover, by definition of $S$, one has $S(W')=W_2$. According to Proposition 
\ref{Psi2}, this means that $W_2\in W(\pi_1\times \pi_2,\psi)^{(n_1)}$. 
\end{proof}

\subsection{Extension of Whittaker functions}\label{extension of Whittaker}

In this section we prove one of the main results of the paper, which is simultaneously a generalization of one part of 
\cite[Corollary of Proposition 1.7]{CPS} and of \cite[Proposition 9.1]{JPSS83}. Both these technical results have proved very useful in the study of Rankin-Selberg $L$-factors. Our generalization will be used in Section \ref{general case} to prove the existence of test vectors for the $L$-factors that we are interested in.  

We recall 
that if $\tau$ is a $P_{r+1}$-submodule of $\Ind_{N_{r+1}}^{P_{r+1}}(\psi)$ with $r\geqslant 1$, such that $\tau^{(1)}$ has a central character, we defined in Proposition \ref{Psi1} a map $S$ from $\tau$ to $\Ind_{N_r}^{G_r}(\psi)$, inducing an injection $\overline{S}$ of $\tau^{(1)}$ in $\Ind_{N_r}^{G_r}(\psi)$. We set 
\[W(\tau^{(1)},\psi)=S(\tau)=\overline{S}(\tau^{(1)}).\] This is consistent with our previous notations, as when $\tau^{(1)}$ is of Whittaker type, $W(\tau^{(1)},\psi)$ is indeed the Whittaker model of 
$\tau^{(1)}$. More generally, for $1\leqslant k \leqslant r$, then $(\Phi^{-})^{k-1}\tau$ is naturally a subspace of 
$\Ind_{N_{r+2-k}}^{P_{r+2-k}}(\psi)$ thanks to Proposition \ref{Phi}, and if $\tau^{(k)}$ has a central character, then 
$\overline{S}$ induces an embedding from $\tau^{(k)}= \Psi^-(\Phi^{-})^{k-1}\tau$ into $\Ind_{N_{r+1-k}}^{G_{r+1-k}}(\psi)$, and we set 
\[W(\tau^{(k)},\psi)=S((\Phi^{-})^{k-1}\tau)=\overline{S}(\tau^{(k)}).\]

\begin{prop}\label{first step}
Let $\tau$ be a $P_{r+1}$-submodule of $\Ind_{N_{r+1}}^{P_{r+1}}(\psi)$ with $r\geqslant 1$, such that $\tau^{(1)}$ has a central character. Then for any $W_0\in W(\tau^{(1)},\psi)$, and any $\phi\in \mathcal{C}_c^\infty(F^r)$, there is $W\in \tau$ such that, for any $g\in G_r$, 
\[W\begin{pmatrix} g & \\ & 1 \end{pmatrix}\phi(\eta_r g)=\d_{P_{r+1}}^{\frac{1}{2}} (g)W_0(g)\phi(\eta_r g).\]
\end{prop}
\begin{proof}
The proof is inspired by the proof of \cite[Proposition 3.2]{M13}. We denote by $c$ the central character 
of $\tau^{(1)}$ and set $\rho=\tau^{(1)}$. Take $W_1\in \tau$ such that $S(W_1)=W_0$, in particular $S(\tau(g)W_1)=\rho(g)W_0$ for all $g$ in $G_r$ (notice that both $\rho$ and $\tau$ act by right translation). For 
$t_1,\dots,t_r \in F^\times$, we will denote by $t$ the element 
\[t=\begin{pmatrix} t_1\dots t_r  &  &  &  \\
 & t_2\dots t_r  &  &  \\ &  & \ddots & \\ &  &  &  t_r \end{pmatrix}\in G_r,\] and set 
\[t'=t_r^{-1}t \in G_r.\] We 
choose representatives $k_1,\dots,k_l$ of $K_r/U$ with $k_1=1$, where $U$ is a compact open subgroup of $K_r$ fixing 
$W_1$ on the right. By the claim in Theorem 2.1 of \cite{M11} (or the arguments in 
the claim of Proposition 1.6 of \cite{CPS}), there is $N_i\in \N$ such that 
\[\tau(k_i)W_1\begin{pmatrix} ta & \\ & 1 \end{pmatrix}=c(a)\nu(a)^{\frac{1}{2}}\tau(k_i)W_1\begin{pmatrix} t & \\ & 1 \end{pmatrix}\] for
any $t_1,\dots, t_{r-1} \in F^\times$, $a\in F^\times$ such that $|a|\leqslant 1$ and $t_r\in F^\times$ 
such that $|t_r| \leqslant q^{-N_i}$. We can choose all $N_i$'s independent of $i$, say equal to an integer $N$, as there is a finite number of them. We set 
$z_b=\diag(b,\dots,b,1)\in P_{r+1}$ and
\[W_b=\frac{\tau(z_b)W_1}{c(b)|b|^{\frac{r}{2}}}.\] 
Notice that for all $i=1,\dots,l$, the functions \[\frac{(\tau(k_i)W_b)\begin{pmatrix} t & \\ & 1 \end{pmatrix}}{c(t_r)|t_r|^{\frac{r}{2}}}
=\frac{\tau(k_i)W_1\begin{pmatrix} t't_rb & \\ & 1 \end{pmatrix}}{c(t_rb)|t_rb|^{\frac{r}{2}}} \] 
are constant with respect to $t_r$ for $|t_rb|\leqslant 1 \Leftrightarrow |t_r|\leqslant q^{-N}/|b|$. 
We write \[\frac{1}{c(t_r)|t_r|^{\frac{r}{2}}}\tau(k_i)W_b\begin{pmatrix} t & \\ & 1 \end{pmatrix}
=  \frac{\nu(t')^{1/2}}{c(t_rb)|t_rb|^{\frac{r}{2}}\nu(t')^{1/2}}\tau(k_i)W_1\begin{pmatrix} t't_rb & \\ & 1 \end{pmatrix}\] 
and observe that $S(\tau(k_i)W_1)=\rho(k_i)W_0.$ Hence by Proposition \ref{Psi1}, making $z=t_rb\rightarrow 0$ we get:
\begin{align*}\frac{(\tau(k_i)W_b)\begin{pmatrix} t & \\ & 1 \end{pmatrix}}{c(t_r)|t_r|^{\frac{r}{2}}}&
= \nu(t')^{1/2}\rho(k_i)W_0(t'),\end{align*} this equality being valid for all 
$t_r$ such that $|t_r|\leqslant q^{-N}/|b|$. So for such $t_r$:
\[(\tau(k_i)W_b)\begin{pmatrix} t & \\ & 1 \end{pmatrix}=|t_r|^{\frac{r}{2}}\nu(t')^{1/2}c(t_r)\rho(k_i)W_0(t')
=\nu(t)^{1/2}\rho(k_i)W_0(t).\]
By the Iwasawa decomposition of $G_r$, we deduce that \[W_b\begin{pmatrix}g & \\ & 1 \end{pmatrix}=\nu(g)^{1/2}W_0(g)\] for all 
$g\in G_r$ such that $\nu(z(g))^{1/r}\leqslant q^{-N}/|b| $. But we can take $b$ such that $q^{-N}/|b|$ is as large as we need, i.e. 
such that $\eta_r g$ belongs to $(\p^{-l})^r$ for $l$ as large as we want. In particular for $\phi\in \mathcal{C}_c^\infty(F^r)$, one can take $l$ large enough for $(\p^{-l})^r$ to contain the support of $\phi$, and this proves the claim. 
\end{proof}

After fixing a Haar measure $dx$ on $F^r$, we denote by $\widehat{\phi}$ the Fourier transform 
of $\phi\in \mathcal{C}_c^\infty(F^r)$ with respect to $dx$ and the character $\psi\otimes \dots \otimes \psi$ of $F^n$.

\begin{corollary}\label{second step}
With same notations as in the Proposition \ref{first step}, for any $W_0\in W(\tau^{(1)},\psi)$, and any $\phi\in \mathcal{C}_c^\infty(F^r)$, there is $W'\in\tau$ such that, for all $g\in G_r$, 
\[W'\begin{pmatrix} g & \\ & 1 \end{pmatrix}=\nu(g)^{\frac{1}{2}}W_0(g)\phi(\eta_r g).\]
\end{corollary}
\begin{proof}
Take $W$ as in the statement of Proposition \ref{first step}. Let $u(x)=\left(\begin{smallmatrix} I_r & x \\ 0 & 1\end{smallmatrix}\right)$. Let $\alpha$ be the element of $\mathcal{C}_c^\infty(F^r)$ is such that $\widehat{\alpha}=\phi$, then if one sets $W'(p)=\int_{x\in F^n} \alpha(x)W(pu(x))dx$
 for $p\in P_{r+1}$, the map $W'$ belongs to 
$\tau$ and \[W'\begin{pmatrix} g & \\ & 1 \end{pmatrix}=W\begin{pmatrix} g & \\ & 1 \end{pmatrix}\widehat{\alpha}(\eta_r g)\] for $g\in G_r$. The statement now follows from Proposition \ref{first step}.
\end{proof}

We thus obtain the following important result. As we said before, it is obviously a generalization of one part of \cite[Corollary of Proposition 1.7]{CPS}, but it is also a generalization of \cite[Proposition 9.1]{JPSS83} thanks to Proposition \ref{inclusion of whittaker model} and Proposition \ref{kirillov}.

\begin{thm}\label{lethal lemma}
Let $\tau$ be a submodule of $\Ind_{N_n}^{P_n}(\psi)$, with $n\geqslant 2$, and let $1\leqslant k \leqslant n-1$ be
such that $\rho=\tau^{(k)}$ has a central character. Then for any 
$W_0\in W(\rho,\psi)$ and $\phi\in \mathcal{C}_c^\infty (F^{n-k})$, there is 
$W\in \tau$ such that \[W\begin{pmatrix} g & \\ & I_k \end{pmatrix}= \nu(g)^{\frac{k}{2}} W_0(g)\phi(\eta_{n-k} g) \] for all $g\in G_{n-k}$.
\end{thm}
\begin{proof}
According to Corollary \ref{second step}, there is $W'\in (\Phi^-)^{(k-1)}(\tau)$ such that for $g\in G_{n-k}$:
\[W'\begin{pmatrix} g & \\ & 1 \end{pmatrix}=\nu(g)^{\frac{1}{2}} W_0(g)\phi(\eta_{n-k} g).\]
Now, Proposition \ref{Phi} repeated $k-1$ times gives the existence of $W\in \tau$ such that 
\[W\begin{pmatrix} g & \\ & I_{k} \end{pmatrix}= \nu(g)^{\frac{k-1}{2}}W'\begin{pmatrix} g & \\ & 1 \end{pmatrix}\] for $g\in G_{n-k}$, the statement follows.
\end{proof}

\section{Test vectors for $L$-factors of pairs of discrete series}

The aim of the second part of this paper is to show that $L(s,\St_l(\rho),\Sp_k(\rho'))=L(s,\St_l(\rho),\St_k(\rho'))$ is given by a single Rankin-Selberg integral. We first recall their definitions.

\subsection{$L$-factors for pairs of discrete series}\label{section pairs of disc}

All results of this section are fundamental facts from \cite{JPSS83}. We normalize the Haar measure on $G_n$ to give volume $1$ to $K_n$, and on any closed subgroup $H$ of $G_n$ we normalize the Haar measure to give volume $1$ to $H\cap K_n$. If $H$ is unimodular, this then defines a unique nonzero right invariant measure 
on $H\backslash G_n$. We consider $\pi\in \rep_\whitt(n)$ and $\pi'\in \rep_\whitt(m)$ with 
$n\geqslant m\geqslant 1$.

If $n=m$, then for $\phi\in \mathcal{C}_c^\infty(F^n)$, $W\in W(\pi,\psi)$ and 
$W'\in W(\pi',\psi^{-1})$, we define for $s\in \C$ the \textit{Rankin-Selberg integral}:
\[I_n(s,W,W',\phi)=\int_{N_n\backslash G_n} W(g)W'(g)\phi(\eta_n g)\nu(g)^sdg.\]
It is absolutely convergent for $\Re(s)$ larger than a real number depending only on $\pi$ and $\pi'$, and extends to an 
element of $\C(q^{-s})$. Moreover, letting $W,W'$ and $\phi$ vary, the subspace of $\C(q^{-s})$ spanned by the integrals 
$I_n(s,W,W',\phi)$ is equal to $L(s,\pi,\pi')\C[q^{\pm s}]$ for a unique Euler factor $L(s,\pi,\pi')$, which is called 
the \textit{$L$-factor} of $\pi$ and $\pi'$.

If $n>m$, then for $W\in W(\pi,\psi)$ and 
$W'\in W(\pi',\psi^{-1})$, we define for $s\in \C$ the \emph{Rankin-Selberg (or Hecke) integral}:
\[I_{n,m}(s,W,W')=\int_{N_m\backslash G_m} W\begin{pmatrix}g & \\ & I_{n-m} \end{pmatrix}W'(g)\nu(g)^{s-\frac{n-m}{2}}dg.\]
It is absolutely convergent for $\Re(s)$ larger than a real number depending only on $\pi$ and $\pi'$, and extends to an 
element of $\C(q^{-s})$. Moreover, letting $W$ and $W'$ vary, the subspace of $\C(q^{-s})$ spanned by the integrals 
$I_{n,m}(s,W,W')$ is equal to $L(s,\pi,\pi')\C[q^{\pm s}]$ for a unique Euler factor $L(s,\pi,\pi')$, which is again called 
the $L$-factor of $\pi$ and $\pi'$. We set $L(s,\pi',\pi)=L(s,\pi,\pi')$.

By definition, following the authors of \cite{JPSS83} again, if $\pi$ and $\pi'$ belong respectively to 
$\Irr(n)$ and $\Irr(m)$, then the Whittaker models 
$W(\S(\pi),\psi)$ and $W(\S(\pi'),\psi^{-1})$ are uniquely determined by $\pi$ and $\pi'$. We set 
$L(s,\pi,\pi')=L(s,\S(\pi),\S(\pi'))$ where $\S(\pi)$ and $\S(\pi')$ are the standard modules over $\pi$ and $\pi'$. 

As explained in the introduction, we will only be interested in the $L$-factors of the form $L(s,\d,\presuper{t}\d')$ for $\d$ and $\d'$ in $\Irr_{\disc}$. 
We recall from Section \ref{notation} that $\d$ is of the form $\St_l(\rho)$ for $\rho\in \Irr_{\cusp}$, similarly 
$\d'$ is of the form $\St_k(\rho')$ for $\rho'\in \Irr_{\cusp}$, hence $\presuper{t}\delta'=\Sp_k(\rho')$ and $\S(\Sp_k(\rho'))=\S_k(\rho')$. As $\d=\S(\d)$ and $\d'=\S(\d')$, we have \[L(s,\d,\d')=L(s,\St_l(\rho),\St_k(\rho'))\] whereas 
\[L(s,\d,\presuper{t}\d')=L(s,\St_l(\rho),\S_k(\rho')).\] We now notice 
that \[L(s,\d,\presuper{t}\d')=L(s,\d,\d'),\] so our integral will also compute $L(s,\d,\d')$. This observation is a consequence of 
\cite[Proposition 8.1, Theorem 8.2 and Proposition 9.4]{JPSS83}.

\begin{prop}\label{L factor discrete}
For $1\leqslant k\leqslant l$ and $\rho$ and $\rho'$ in $\Irr_{\cusp}$, the $L$ factor $L(s,\St_l(\rho),\St_k(\rho'))$ is equal to $1$ unless there is $s_0\in \C$ such that $\rho'\simeq \nu^{s_0}\rho^\vee$.  In the case $\rho'\simeq \nu^{s_0}\rho^\vee$, one has the equalities
\begin{align*}
L(s,\St_l(\rho),\St_k(\rho'))&=\prod_{i=1}^k L(s+s_0,\nu^{\frac{l-1}{2}}\rho,\nu^{\frac{k+1-2i}{2}}\rho^\vee)\\
&=L(s,\St_l(\rho),\S_k(\rho')) \\~&=L(s,\St_l(\rho),\Sp_k(\rho')).
\end{align*}
\end{prop}

\begin{rem} \label{remLtriv}
If $\rho$ and $\rho'$ are not equal up to an unramified twist, then it is not new that the Rankin-Selberg $L$-factor $L(s,\St_l(\rho),\S_k(\rho'))=1$ is given by a single integral: this follows from 
the proof of \cite[Theorem 2.7]{JPSS83}. If they are equal up to unramified twist, it is enough to show the test vector result 
for $\rho'=\rho^\vee$. \end{rem}
Following Remark \ref{remLtriv}, henceforth we can and will assume that $\rho'=\rho^\vee$. We will also suppose that $\psi$ has conductor $0$, i.e. is trivial on $\o$, but not on $\p^{-1}$.

Let $r$ be the integer such that $\rho\in \Irr_{\cusp}(r)$. If $l=k$ we will find $W\in W(\St_l(\rho),\psi)$, $W'\in W(\S_k(\rho^\vee),\psi^{-1})$, and $\phi\in \mathcal{C}_c^\infty(F^n)$ such that $I_{lr}(s,W,W',\phi)=L(s,\St_l(\rho),\St_k(\rho^\vee))$, whereas if $l>k$, we will find $W\in W(\St_l(\rho),\psi)$ 
and $W'\in W(\S_k(\rho^\vee),\psi^{-1})$ such that $I_{lr,kr}(s,W,W')=L(s,\St_l(\rho),\St_k(\rho^\vee))$. Hence it is fair to say that we find test vectors for $L(s,\St_l(\rho),\Sp_k(\rho^\vee))$ rather then $L(s,\St_l(\rho),\St_k(\rho^\vee))$ although the $L$-factors are equal. 

Taking $W'\in W(\S_k(\rho^\vee),\psi^{-1})$ makes things simpler,  as the space $W(\S_k(\rho^\vee),\psi^{-1})$ contains $W(\St_k(\rho^\vee),\psi^{-1})$ as a proper subspace. We believe that it is possible to take $W'\in W(\St_k(\rho^\vee),\psi^{-1})$, but it seems much more difficult to us.  To justify and motivate the fact that allowing to take $W\in W(\S_k(\rho),\psi^{-1})$ simplifies matters, we start with the toy example $\rho=\mathbf{1}$.

\subsection{The case of Steinberg representations}\label{steinberg case}

Here $\rho=\mathbf{1}$ (the trivial character of $G_1$). We set $\S_k=\S_k(\1)$, $\St_k=\St_k(\1)$, and $\Sp_k=\Sp_k(\1)=\mathbf{1}_{G_k}$. For $k=2$, by \cite[Proposition 5.3.7]{KM10}, there are $W\in W(\St_2,\psi)$ and $W'\in W(\St_2,\psi^{-1})$, and $\phi\in \mathcal{C}_c^\infty(F^2)$ such 
 that $I_2(s,W,W',\phi)=L(s,\St_2,\St_2)$, and it is in fact shown in [ibid.] that one can take $W$ and $W'$ to be the essential vectors in $W(\St_2,\psi)$ and $W(\St_2,\psi^{-1})$. The proof of this result is already quite technical.
 Notice that $W(\S_k,\psi^{-1})$ is spherical (contains $K_n$-fixed vectors), whereas the smaller space $W(\St_k,\psi^{-1})$ is not spherical when $k\geqslant 2$. Hence we choose $W_k^0$ to be the normalized spherical vector in $W(\S_k,\psi^{-1})$. 
 On the other hand, in $W(\St_k,\psi)$, we choose the essential vector $W_k^{\ess}$ of \cite{JPSS81} (see also \cite{J12}, \cite{M13} and \cite{M14}). We fix $l\geqslant k \geqslant 1$, and moreover as for 
 $l=k=1$ one has $W_1^{\ess}=W_1^0=\1$, Tate's thesis gives us the equality \[I_1(s,W_1^{\ess}, W_1^0,\1_\o)=L(s,\St_1,\St_1),\] 
 we suppose that $l\geqslant 2$. 

If $l>k$, then Corollary 3.3 of \cite{M13} immediatly gives the existence of test vectors.  

\begin{prop}\label{different steinberg}
If $l>k\geqslant 1$ then \[I_{l,k}(s,W_l^{\ess}, W_k^0)=L(s,\St_l,\Sp_k)=L(s,\St_l,\St_k).\]
\end{prop}
 
We now consider the case $l=k\geqslant 2$. 
Take $\phi$ a Schwartz function on $F^l$ of the form 
\[\phi=\l \mathbf{1}_{(\p^f)^{l-1} \times (1+\p^f)},\]
with $f$ sufficiently large, and $\l=\mu(1+\p^f)^{-1}$ for $d\mu$ the normalized (see the beginning of Section \ref{section pairs of disc}) Haar measure on $F^\times$, then 
\[I_l(s,W_l^{\ess},W_l^0,\phi)= \int_{N_{l-1}\backslash G_{l-1}}W_l^{\ess}\begin{pmatrix} g & \\ & 1 \end{pmatrix}
W_l^0\begin{pmatrix} g & \\ & 1 \end{pmatrix}\nu(g)^{s-1}dg.\]

The integral on the right above (a formula which makes sense only for 
$\Re(s)$ large) is also considered as a rational function of $q^{-s}$.

In \cite[Definition 1.3]{M13}, to each $\pi\in \Irr_{\gen}$, one attaches an unramified standard module $\pi_u$ as follows: 
let $\{\St_{k_i}(\mu_i)\}_{i=1,\dots,b}$ be the multiset of Steinberg representations with $\mu_i$ an unramified character of $F^\times$ occurring as factors of the product of discrete series 
$\pi$ (see Section \ref{notation}) numbered such that 
\[e(\nu^{\frac{k_1-1}{2}}\mu_1)\geq  \dots \geq  e(\nu^{\frac{k_b-1}{2}}\mu_b),\] then 
\[\pi_u:=\nu^{\frac{k_1-1}{2}}\mu_1\times \dots \times \nu^{\frac{k_b-1}{2}}\mu_b.\] 
For $\pi=\St_l$, one gets: $(\St_l)_u=\nu^{\frac{l-1}{2}}$. We write $W_{(\St_l)_u}^0=\nu^{\frac{l-1}{2}}$ for its normalized spherical vector. We recall the formula for $W_l^{\ess}$ given in \cite[Corollary 3.2]{M13}. Take $n\in N_{l-1}$, 
$a=\diag(a_1,\dots,a_{l-1})\in G_{l-1}\subset G_l$, and $k\in K_{l-1}\subset G_{l-1}\subset G_l$:

\[W_l^{\ess}\begin{pmatrix} nak & \\ & 1 \end{pmatrix}=\psi(n)\nu(a_1)^{l-1}\1_\o(a_1)\prod_{i=2}^{l-1}\1_{\o^\times}(a_i).\]

For $H$ a closed subgroup of $G_l$, we denote by $\d_H$ its modulus character which satisfies that if $dh$ is a right invariat Haar measure on $H$, then for any $f\in \mathcal{C}_c^\infty(H)$, and any $x\in H$: $dh (f(x^{-1} \ .))=\d_H(x)dh(f)$. Now we do computations similar to those in the proof of \cite[Corollary 3.3, Case $m=r$]{M13}. 

\begin{align*}
\int_{N_{l-1}\backslash G_{l-1}}&W_l^{\ess}\begin{pmatrix} g & \\ & 1 \end{pmatrix}W_l^0\begin{pmatrix} g & \\ & 1 \end{pmatrix}\nu(g)^{s-1}dg\\
&=\int_{(F^\times)^{l-1}}W_l^0\begin{pmatrix} a & \\ & 1 \end{pmatrix}|a_1|^{s+l-2}\1_\o(a_1)\prod_{i=2}^{l-1}\1_{\o^\times}(a_i)|a_i|^{s-1}\d_{B_{l-1}}^{-1}(a)\prod_{i=1}^{l-1} d_{F^\times}a_i\\
&=\int_{G_1}W_l^0\begin{pmatrix} a_1 & \\ & I_{l-1} \end{pmatrix}|a_1|^{s+l-2}\1_\o(a_1)\d_{B_{l-1}}^{-1}\begin{pmatrix} a_1  & \\ & I_{l-2}\end{pmatrix}d_{F^\times}a_1 \\
&=\int_{G_1}W_l^0\begin{pmatrix} a_1 & \\ & I_{l-1} \end{pmatrix}|a_1|^{s}d_{F^\times}a_1\\ 
&\text{(because }W_l^0(a_1)=W_l^0(a_1)\1_\o(a_1)\text{ by \cite{S76})}\\
&=I_{l,1}(s,W_l^0,\nu^{\frac{l-1}{2}})=\prod_{i=1}^{l} L(s,\nu^{\frac{1}{2}},\nu^{\frac{l+1-2i}{2}})\\&=L(s,\St_l,\Sp_l)=L(s,\St_l,\St_l),\end{align*}
the antepenultimate equality follows from \cite[Proposition 2.3]{JS81I}, and \cite[Section 1, equality (3)]{JS81II} (or more precisely their immediate extensions to standard modules as explained in the discussion of \cite[p. 1201]{M13}), and the two last equalities by Proposition \ref{L factor discrete}. Hence we just proved:

\begin{prop}\label{equal steinberg}
For $l\geqslant 2$, then for $f$ large enough, and a well-chosen multiple $\phi$ of 
$\mathbf{1}_{(\p^f)^{l-1} \times (1+\p^f)}$, we have
\[I_l(s,W_l^{\ess},W_l^0,\phi)= L(s,\St_l,\Sp_l)= L(s,\St_l,\St_l).\]
\end{prop}

\subsection{The general case}\label{general case}
In this section, for $l\geqslant k\geqslant 1$, we find test vectors for $L(s,\St_l(\rho),\Sp_k(\rho^{\vee}))$ by reducing the problem to the known case of pairs of cuspidal representations (\cite[Theorem 9.1]{KM17}) thanks to Theorem \ref{lethal lemma}.

We fix $\rho\in \Irr_{\cusp}(r)$ for $r\geqslant 1$. We set $n=lr$ and $m=kr$. By \cite[Theorem 4.4 and Lemma 4.5]{BZ77} and \cite[Proposition 9.6]{Z80}, using moreover that representations of $G_r$ 
with different central characters are always in direct sum, we have the formulae 
\begin{align} \label{derivative standard}
\S_k(\rho^\vee)^{((k-1)r)}&\simeq \nu^{\frac{1-k}{2}}\rho^\vee\oplus \cdots \oplus \nu^{\frac{k-1}{2}}\rho^\vee,\\
 \label{derivative steinberg}
\St_k(\rho^\vee)^{((k-1)r)}&\simeq \nu^{\frac{k-1}{2}}\rho^\vee.
\end{align}

Let us first treat separately the easy case $k=1$. Then $\S_1(\rho^\vee)=\St_1(\rho^\vee)=\rho^\vee$. In this case, by \cite[Theorem 9.1]{KM17}, there is 
$W^{\cusp}\in W(\nu^{\frac{l-1}{2}}\rho,\psi)$, $V^{\cusp}\in W(\rho^{\vee},\psi^{-1})$, and $\phi\in \mathcal{C}_c^\infty (F^r)$ such that 
\begin{equation}\label{cuspidal one} I_r(s,W^{\cusp},V^{\cusp},\phi)=L(s,\nu^{\frac{l-1}{2}}\rho,\rho^\vee)=L(s,\St_l(\rho),\rho^\vee)
\end{equation}

Now by Theorem \ref{lethal lemma} applied to $\Res_{P_n}(W(\St_l(\rho),\psi))$ (\cite[Proposition 9.1]{JPSS83} is in fact sufficient here), there is 
$W_l\in W(\St_l(\rho),\psi)$ such that \[W_l\begin{pmatrix} g & \\ & I_{n-r}\end{pmatrix}=\nu(g)^{\frac{n-r}{2}}W^{\cusp}(g)\phi(\eta_r g)\] for all $g\in G_r$. The following proposition follows at once from Equation (\ref{cuspidal one}).

\begin{prop}\label{one cusp}
For $l\geqslant 1$, there is $W_l\in W(\St_l(\rho),\psi)$ and $V^{\cusp}\in W(\rho^\vee,\psi^{-1})$ (as above), such that 
\[I(s,W_l,V^{\cusp})=L(s,\St_l(\rho),\rho^\vee).\]
\end{prop}

Hence now we suppose that $k\geqslant 2$ until the end of the paper. According to \cite[Theorem 9.1]{KM17} again, one can find $W^{\cusp}\in W(\nu^{\frac{l-1}{2}}\rho,\psi)$, $V_i^{\cusp}\in W(\nu^{\frac{1-k+2i}{2}}\rho^\vee,\psi^{-1})$ (one can in fact 
take $V_i^{\cusp}=\nu^{\frac{1-k+2i}{2}}V^{\cusp}$ for $V^{\cusp}$ as above) for each $i$ between $0$ and $k-1$, and 
$\phi\in \mathcal{C}_c^\infty (F^r)$ such that 
\begin{equation}\label{cuspidal case}  I_r(s,W^{\cusp},V_i^{\cusp},\phi)=L(s,\nu^{\frac{l-1}{2}}\rho,\nu^{\frac{1-k+2i}{2}}\rho^\vee).\end{equation}

By Theorem \ref{lethal lemma} and Equation (\ref{derivative standard}), 
we deduce that for each $i$ between $0$ and $k-1$, there is $V_i\in W(\S_k(\rho^\vee),\psi^{-1})$ such that 

\[ I_r(s,W^{\cusp},V_i^{\cusp},\phi)=I_{kr,r}(s,V_i,W^{\cusp}),\] 

which together with Equation (\ref{cuspidal case}) yields 

\begin{equation}\label{on the way} I_{kr,r}(s,V_i,W^{\cusp})=L(s,\nu^{\frac{l-1}{2}}\rho,\nu^{\frac{1-k+2i}{2}}\rho^\vee).
\end{equation}

We recall that if $d$ is the cardinality of the group $R(\rho)$ of unramified characters fixing $\rho$, and denoting by $\varpi$ 
a uniformizer of $F$, then $\chi\mapsto \chi(\varpi)$ is an isomorphism between $R(\rho)$ and the group of 
$d$-th roots of unity in $\C^\times$ (notice that \cite[Lemma 6.2.5]{BK93} gives an "arithmetical" description of 
$d$ as $d=r/e$ where $e$ is the ramification index of $\rho$ in the sense of type theory). In particular 
\[\prod_{\chi\in R(\rho)}(1-\chi(\varpi)X)=1-X^d.\] Setting $X=q^{-s}$, we have 
\[L(s,\nu^{\frac{l-1}{2}}\rho,\nu^{\frac{1-k+2i}{2}}\rho^\vee)= L(s+\frac{l-k+2i}{2},\rho,\rho^\vee),\] 
which according to \cite[Proposition 8.1]{JPSS83} and the discussion above is equal to: 
\[\frac{1}{1-q^{\frac{(k-l-2i)d}{2}}X^d}.\] 

But as a function 
of $Y=X^d$, one has 
\[L(s,\St_l(\rho),\St_k(\rho^\vee))=\prod_{0=1}^{k-1} \frac{1}{1-q^{\frac{(k-l-2i)d}{2}}Y},\] and, in particular, it has simple poles. We can thus write its partial fraction decomposition, and 
find (explicit) $\l_i\in \C$ such that
\begin{align}
\notag L(s,\St_l(\rho),\St_k(\rho^\vee))&=\sum_{i=0}^{k-1}\frac{\l_i}{1-q^{\frac{(k-l-2i)d}{2}}Y}= 
\sum_{i=0}^{k-1}\l_iL(s,\nu^{\frac{l-1}{2}}\rho,\nu^{\frac{1-k+2i}{2}}\rho^\vee)\\
\label{key}&=\sum_{i=0}^{k-1} \l_i I_{kr,r}(s, V_i,W^{\cusp})=I_{kr,r}(s,\sum_{i=0}^{k-1} \l_i V_i,W^{\cusp}),\end{align}
the second to last equality thanks to Equation (\ref{on the way}).

Set $n=lr$ and $m=kr$, we are now in position to prove the second and last main result of the paper.

\begin{thm}\label{TV2}
If $l>k\geqslant 2$, there is $W_l$ in $W(\St_l(\rho),\psi)$, and $W'_k\in W(\S_k((\rho^\vee)),\psi^{-1})$, such that 
\[I_{n,m}(s,W_l,W'_k)=L(s,\St_l(\rho),\Sp_k(\rho^\vee))=L(s,\St_l(\rho),\St_k(\rho^\vee)).\] 
If $l=k\geqslant 2$,  there is $W_l$ in $W(\St_l(\rho),\psi)$, $W'_l\in  W(\S_l(\rho^\vee),\psi^{-1})$, and 
$\phi \in \mathcal{C}_c^\infty (F^n)$, such that 
\[I_{n}(s,W_l,W'_l,\phi)=L(s,\St_l(\rho),\Sp_l(\rho^\vee))=L(s,\St_l(\rho),\St_l(\rho^\vee)).\] 
Moreover one can always choose $\phi=\mathbf{1}_{(\p^f)^{n-1} \times (1+\p^f)}$ for $f$ large enough.
\end{thm}
\begin{proof}
Let's first deal with the $l>k$ case. Set $W^\cusp\in W(\nu^{\frac{l-1}{2}}\rho,\psi)$ and $\phi\in \mathcal{C}_c^\infty(F^r)$ as in Equation (\ref{cuspidal one}). By Theorem \ref{lethal lemma}, we can find 
$W\in W(\St_l(\rho),\psi)$ such that 
\begin{equation}\label{lethal again}
W\begin{pmatrix} g & \\ & I_{n-r}\end{pmatrix}=\nu(g)^{\frac{n-r}{2}} W^{\cusp}(g)\phi(\eta_r g)
\end{equation}
for $g\in G_r$. We write $\mathcal{M}_{m-r,r}$ for the additive group of $m-r\times m$ matrices with coefficients in $F$, and $B_{m-r}^-$ for the subgroup of lower 
triangular matrices in $G_{m-r}$. Then by 
\cite[Lemma 9.2]{JPSS83}, there is $W'\in W(\St_l(\rho),\psi)$ such that 
\begin{equation}\label{trick from JPSS}
W'\begin{pmatrix} g &  & \\ x & h & \\ &  & I_{n-m} \end{pmatrix}=W\begin{pmatrix} g & \\ &  & I_{n-r} \end{pmatrix}
\end{equation}
 for $x$ 
and $h$ in open compact subgroups $U_1$and $U_2$ of $\mathcal{M}_{m-r,r}$ and $B_{m-r}^-$ respectively that we can choose as small as we like, and is equal to zero for $x$ or $h$ outside of those subgroups. We choose $U_1$ and $U_2$ small enough such that 
matrices $\left(\begin{smallmatrix} I_r &  \\ x & h \end{smallmatrix}\right)$ for $x\in U_1$ and $h\in U_2$ fix all functions $V_i$ by right translation.
The following integration formula for positive measurable functions $F$ on $N_m\backslash G_m$
\begin{equation}\label{integration formula} \int_{N_m \backslash G_m} F(g) dg = \int_{g\in N_r\backslash G_r, x\in \mathcal{M}_{m-r,r}, h\in B_{m-r}^-} F\begin{pmatrix} g &  \\ x & h \end{pmatrix}\nu(g)^{(r-m)} dx d_rh dg
\end{equation}
is valid for an appropriate right Haar measure $d_rh$ on $B_{m-r}^-$. In combination with Equations (\ref{lethal again}) and (\ref{trick from JPSS}), it gives the equality 
\begin{equation}\label{almost} I_{n,m}(s,W',V_i)=I_{m,r}(s,V_i,W^{\cusp}).
\end{equation}
Hence setting \[V'=\sum_{i=0}^{k-1} \l_i V_i\in W(\S_k(\rho^\vee),\psi^{-1}),\] 
Equations (\ref{key}) and (\ref{almost}) together imply the equality
 \[I_{n,m}(s,W',V')= L(s,\St_l(\rho),\Sp_k(\rho^\vee)).\]
We are thus done in this case. Notice that our choices for $W'$ and $V'$ are highly non-canonical.

It remains to deal with the case $l=k\geqslant 2$ (i.e. $n=m\geqslant 2r$), and thanks to Section \ref{steinberg case}, we suppose that $r\geqslant 2$. We take $V'$ as above, and thanks to 
\cite[Lemma 9.2]{JPSS83} again, we take $W'\in W(\St_l(\rho),\psi)$ such that for $g\in G_r$:
\begin{equation}\label{another trick from JPSS}
W'\begin{pmatrix} g &  & \\ x & h \\ & & I_{n-m} \end{pmatrix}=W\begin{pmatrix} g & \\ & I_{n-r} \end{pmatrix}
\end{equation}
 for $x$ 
and $h$ in open compact subgroups $U_1$ and $U_2$ of respectively $\mathcal{M}_{m-1-r,r}$ and $B_{m-1-r}^-$ and is equal to zero for $x$ or $h$ outside of those subgroups. Again we choose $U_1$ and $U_2$ small enough such that 
matrices $\left(\begin{smallmatrix} I_r &  & \\ x & h & \\ & & 1\end{smallmatrix}\right)$ for $x\in U_1$ and $h\in U_2$ fix all functions $V_i$ by right translation. If $\phi$ is a Schwartz function of the form 
$\l\1_{(\p^f)^{n-1} \times (1+\p^f)}$ for 
$f$ large enough and $\l=\mu(1+\p^f)^{-1}$for $d\mu$ the normalized Haar measure on $F^\times$, we have: 
\[I_m(s,W',V',\phi)=\int_{N_{m-1}\backslash G_{m-1}}W'\begin{pmatrix} g & \\ & 1 \end{pmatrix}V'\begin{pmatrix} g & \\ & 1 \end{pmatrix}\nu(g)^{s-1}dg.\]

The the integration formula (\ref{integration formula}) with $m-1$ instead of $m$ together with our choices of $W'$ and $V'$ gives 
the equality 
 
\[\int_{N_{m-1}\backslash G_{m-1}}W'\begin{pmatrix} g & \\ & 1 \end{pmatrix}V'\begin{pmatrix} g & \\ & 1 \end{pmatrix}\nu(g)^{s-1}dg=I_{m,r}(s,\sum_{i=0}^{k-1} \l_i V_i,W^{\cusp}),\] 
and we conclude again, by appealing to Equality (\ref{key}), that 
\[I_m(s,W',V',\phi)=L(s,\St_l(\rho),\Sp_l(\rho^\vee)).\]
\end{proof}

\bibliographystyle{plain}
\bibliography{TV2}
\end{document}